\documentclass[12pt]{article}
\usepackage[utf8]{inputenc}
\usepackage[T1]{fontenc}
\usepackage[frenchb,english]{babel} 
\usepackage{textcomp}
\usepackage{amsmath,amssymb}
\usepackage{amsthm}
\usepackage{lmodern}

\usepackage[a4paper,margin=2cm]{geometry}

\usepackage{graphicx}             
\usepackage{xcolor}               
\usepackage{microtype,stmaryrd}          

\usepackage{hyperref}
\hypersetup{pdfstartview=XXZ}

\usepackage{bbm}
\usepackage{mathrsfs}
\usepackage{subfig}

\def\build#1_#2^#3{\mathrel{
\mathop{\kern 0pt#1}\limits_{#2}^{#3}}}
\def\llbracket{[\hspace{-.10em} [ }
\def\rrbracket{ ] \hspace{-.10em}]}

\newtheorem{theorem}{Theorem}
\newtheorem{proposition}[theorem]{Proposition}

\newtheorem{lemma}[theorem]{Lemma}

\def\w{\mathrm{w}}

\def\t{\mathcal{T}}
\def\b{\mathcal{B}}

\def\r{\mathcal{R}}

\def\S{\mathcal{S}}
\def\T{\mathbb{T}}

\def\P{\mathbb{P}}

\def\R{\mathbb{R}}

\def\cc{\mathcal{C}}
\def\ddd{\mathcal{D}}

\def\ve{{\varepsilon}}
\def\la{\longrightarrow}

\def\ov{\overline}

\def\dd{\mathrm{d}}

\def\ll{\mathcal{L}}

\def\jj{\mathcal{J}}

\def\bb{\mathbf{b}}

\def\be{\mathbf{e}}

\def\rem{\noindent{\bf Remark. }}

\author{Jean-Fran\c cois Le Gall\footnote{e-mail: jean-francois.le-gall@universite-paris-saclay.fr}} 
\title{Large deviations for the CRT}
\date{\small Universit\'e Paris-Saclay}

\begin{document}
\maketitle

\begin{abstract}
We establish large deviation principles both for the CRT in the Gromov-Hausdorff topology and
for the CRT weighted by its mass measure in the Gromov-Hausdorff-Prohorov topology. In the first case,
the rate function of a tree $\t$ is half the square of the length of $\t$. In the second case, the rate function of a 
pair $(\t,\mu)$ is half the integral with respect to the length measure of $\t$ of the inverse of the absolutely continuous
part of $\mu$. 
\end{abstract}

\section{Introduction}

The Brownian Continuum Random Tree, or CRT in short, was introduced by Aldous \cite{Al} in 1991. It has become since a central
object in probability theory, as it appears in many different settings, and in particular in scaling limits for discrete random structures,
including random trees, random graphs and other combinatorial models. For many
purposes, it is convenient to view the CRT as the tree $\t_{\be}$ coded by a normalized Brownian excursion $\be=(\be_t)_{0\leq t\leq 1}$ \cite{Al2}. We may consider the CRT $\t_{\be}$ as a random variable 
with values in the space $\T$ of equivalence classes modulo isometries of rooted compact $\R$-trees, which is 
equipped with the Gromov-Hausdorff topology (see Section \ref{sec:preli} 
for precise definitions). 

The main goal of this work is to investigate large deviation results for the CRT. In order to state our first theorem, we need to introduce 
the notion of the length measure of a rooted compact $\R$-tree. For any $\t\in\T$, we can define a $\sigma$-finite measure $\lambda_\t$
which puts no mass on the set of leaves of $\t$, and is such that the measure $\lambda_\t(I)$ of any segment $I$ of the tree is equal
to the length of this segment. The length of the tree $\t$ is then defined by $L(\t)=\lambda_\t(\t)$. It is well known that the length of the CRT is a.s. infinite (this follows for instance from the stick-breaking construction
of the CRT). 

For $\ve>0$, we write $\ve\cdot \t_{\be}$ for the tree obtained by multiplying all distances in $\t_{\be}$ by the factor $\ve$. Equivalently,
$\ve\cdot \t_{\be}$ is the tree coded by $\ve\,\be$.

\begin{theorem}
\label{LDP-GH}
The laws of $\ve\cdot \t_{\be}$, $\ve >0$, satisfy a large deviation principle with speed $\ve^{-2}$ and
good rate function
$$I(\t)=2\,L(\t)^2.$$
\end{theorem}

Recall that the statement of the theorem means that, for any open subset $O$ of $\T$,
$$\liminf_{\ve\to 0} \ve^2\,\log \P(\ve\cdot \t_{\be}\in O)\geq -2 \inf_{\t\in O} L(\t)^2,$$
and, for any closed subset $F$ of $\T$,
$$\limsup_{\ve\to 0} \ve^2\,\log \P(\ve\cdot \t_{\be}\in F)\leq -2 \inf_{\t\in F} L(\t)^2.$$

Theorem \ref{LDP-GH} is of course reminiscent of Schilder's theorem giving a large deviation principle
for Brownian motion. Indeed, the proof relies on a version of Schilder's theorem for the normalized Brownian
excursion, and then an application of the contraction principle, using the continuity of the
mapping $\be\mapsto \t_{\be}$. The main issue in the proof is to identify the rate function $I$.

In many applications, it is important to view the CRT as equipped with a ``uniform probability measure'', also called the mass measure of the tree.
The coding of $\t_{\be}$ by the Brownian excursion $\be$ involves a mapping $t\mapsto p_{(\be)}(t)$ from $[0,1]$ onto $\t_{\be}$, which corresponds to
a ``contour exploration'' of the tree, and the mass measure $\mu_{\be}$ is then the pushforward of Lebesgue measure on $[0,1]$ under $p_{(\be)}$
($\mu_\be$ can be defined in many other ways).
In contrast with the length measure, the mass measure is supported on the set of leaves of the tree $\t_\be$.

We can then consider the pair $(\t_{\be},\mu_{\be})$ as a random variable with values in the space $\T_{\w}$ of all
weighted trees, or more precisely rooted compact $\R$-trees given with a probability measure. The space $\T_{\w}$ is
equipped with the Gromov-Hausdorff-Prohorov topology. 

\begin{theorem}
\label{Schilder-measure}
The laws of $(\t_{\ve\be}, \mu_{\ve \be})$, $\ve>0$, satisfy a large deviation principle with speed $\ve^{-2}$ and good
rate function $J(\t,\mu)$ specified as follows.
\begin{itemize}
\item[$\bullet$] If $\t=\{\rho\}$ is the trivial tree consisting only of the root,
$J(\t,\delta_{\rho})=0.$
\item[$\bullet$] If $\t$ has finite (positive) length,
\begin{equation}
\label{key-formu}
J(\t,\mu)=2 \int_\t \frac{1}{q_\t(x)}\,\lambda_\t(\dd x)
\end{equation}
where $q_\t(x)\,\lambda_\t(\dd x)$ is the absolutely continuous part in the Lebesgue decomposition 
of $\mu$ with respect to the length measure $\lambda_\t$;
\item[$\bullet$] If $\t$ has infinite length, $J(\t,\mu)=\infty$.
\end{itemize}
\end{theorem}

\rem To make the connection with Theorem \ref{LDP-GH}, note that, when $0<L(\t)<\infty$, the Cauchy-Schwarz inequality
gives
$$L(\t)\leq \Big(\int \frac{1}{q_\t(x)}\lambda_\t(\dd x)\Big)^{1/2}\Big(\int q_\t(x)\lambda_\t(\dd x)\Big)^{1/2}\leq (\frac{1}{2}\,J(\t,\mu))^{1/2}$$
so that $J(\t,\mu)\geq 2\,L(\t)^2$, and this bound is an equality when $\mu=L(\t)^{-1}\,\lambda_\t$. 

\smallskip

Again Theorem \ref{Schilder-measure} follows from Schilder's theorem and an application of the
contraction principle. The identification of the rate function $J$ is however more delicate and requires a
couple of analytic lemmas. 

The paper is organized as follows. Section \ref{sec:preli} recalls basic facts about compact $\R$-trees
and their coding by functions. Theorem \ref{LDP-GH} is proved in Section \ref{sec:LDP1}. Section \ref{sec:anal}
proves the two technical lemmas that are needed in the proof of Theorem \ref{Schilder-measure}. Finally,
Section \ref{sec:LDP2} is devoted to the proof of Theorem \ref{Schilder-measure}.

\section{Preliminaries}
\label{sec:preli}

In this section, we recall the basic facts about compact $\R$-trees that we will use.

\subsection{Compact $\R$-trees}

A compact $\R$-tree is a compact metric space $(\t,d_\t)$ such that the following properties hold for every distinct $x,y\in\t$.
\begin{itemize}
\item[\rm(i)] There is a unique isometric map $f_{x,y}$ from $[0,d_\t(x,y)]$ into $\t$ such that $f_{x,y}(0)=x$ and $f_{x,y}(d_\t(x,y))=y$.
\item[\rm(ii)] If $g$ is a continuous and injective map from $[0,1]$ into $\t$ such that $g(0)=x$ and $g(1)=y$ then $g([0,1])=f_{x,y}([0,d_\t(x,y)])$.
\end{itemize}
The range of the map $f_{x,y}$ in (i) will be denoted by $\llbracket x,y\rrbracket$ and called the segment between $x$ and $y$ (by (i), this segment is isometric to the line interval 
$[0,d_\t(x,y)]$).
By definition, the interior of this segment is $\rrbracket x,y\llbracket=\llbracket x,y\rrbracket \backslash\{x,y\}$. 
%, and we also use the obvious notation $\llbracket x,y\llbracket$. 

We will consider rooted compact $\R$-trees, meaning that there is a distinguished point $\rho_\t$ of $\t$ called the root of $\t$. If $x,y\in \t$
are such that $x\in\llbracket \rho_\t,y\rrbracket$, we say that $y$ is a descendant of $x$. We note the following simple fact. If $x,y\in \t$,
there is a unique point $x\wedge y$ such that $\llbracket \rho_\t,x\rrbracket\cap\llbracket \rho_\t,y\rrbracket=\llbracket \rho_\t,x\wedge y\rrbracket$,
and $d_\t(x,y)=d_\t(x\wedge y,x)+d_\t(x\wedge y,y)$. 

In what follows, we will say tree instead of rooted compact $\R$-tree. We write $\T$ for the space of all trees modulo root-preserving isometries, which is
equipped with the (pointed) Gromov-Hausdorff distance $d_{GH}$. 

Let $\t\in\T$. A point $x\in\t$ is a branching point if $\t\backslash\{x\}$ has at least three connected components. By convention, we also say that
the root $\rho_\t$ is a branching point if $\t\backslash\{\rho_\t\}$ has at least two connected components. We say that $x\in\t$
is a leaf if $\t \backslash\{x\}$ is connected (with this definition, $\rho_\t$ may be a leaf). We write $\b_\t$ for the set of all branching points of $\t$ and
$\ll_\t$ for the set of all leaves of $\t$. We define the skeleton $\t^\circ$ of $\t$ by $\t^\circ=\t\backslash \ll_\t$. 

It will be useful to consider simple trees. 
A tree $\t$ is a simple tree if it is the union of a finite number of
segments. Both $\b_\t$ and $\ll_\t$ are then finite sets. For a simple tree $\t$, a segment $S$ of $\t$ is called 
elementary  if it is maximal (for the inclusion relation) among all segments whose interior contains no branching point. Note that the interiors of elementary segments 
are disjoint and cover $\t\backslash(\b_\t\cup\ll_\t)$. 

Any tree is the limit of a sequence of simple trees. In fact, let $\t$ be a tree and let $(x_n)_{n\geq 1}$ be a dense sequence in $\t$, and, 
for every $n\geq 1$, set
\begin{equation}
\label{simple-approx}
\t_{(n)}=\llbracket \rho_\t,x_1\rrbracket \cup \llbracket \rho_\t,x_2\rrbracket \cup \cdots \cup \llbracket \rho_\t,x_n\rrbracket.
\end{equation}
Then, $\t_{(n)}$ equipped with the distance induced by $d_\t$ is a simple tree, and a compactness argument shows that the maximal distance from a point of $\t$
to $\t_{(n)}$ tend to $0$ as $n\to\infty$, and thus $\t_{(n)}$ converges to $\t$ as $n\to\infty$
in the Hausdorff sense (and a fortiori in the Gromov-Hausdorff sense).   

Let $\t\in \T$. Then, we can define a $\sigma$-finite measure $\lambda_\t$ on $\t$, called the length measure on $\t$, which is supported on the skeleton $\t^\circ$ and such that, for any segment $\llbracket x,y\rrbracket$ of $\t$,
$\lambda_\t(\llbracket x,y\rrbracket)=d_\t(x,y)$ (see e.g. Section 4.3.5 in \cite{Evans}). In the case of a simple tree, the measure $\lambda_\t$
is constructed as the sum of Lebesgue measures on the elementary segments of $\t$. Then, for a general tree $\t$ in $\T$, $\lambda_\t$ is
the increasing limit of the measures $\lambda_{\t_{(n)}}$, where the sequence $\t_{(n)}$ is constructed as in \eqref{simple-approx}. Alternatively, we
may define $\lambda_\t$ as the one-dimensional Hausdorff measure on the skeleton $\t^\circ$. 

By definition, the length of the tree $\t$ is $L(\t):=\lambda_\t(\t)\leq \infty$. For the sequence $(\t_{(n)})_{n\geq 1}$ defined in \eqref{simple-approx}, 
we have $L(\t)=\lim\uparrow L(\t_{(n)})$ as $n\to\infty$. 

\begin{figure}[!h]
 \begin{center}
 \includegraphics[width=14cm]{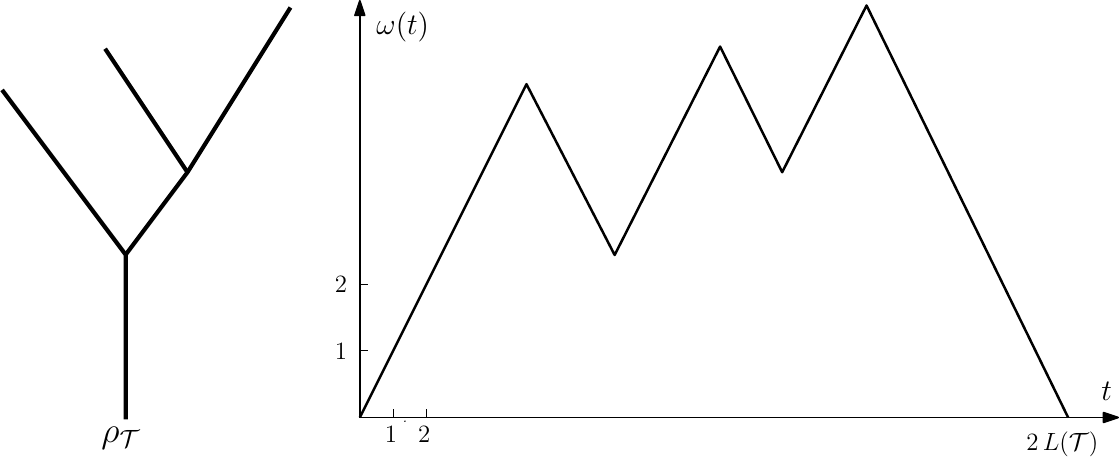}
 \caption{\label{t-coding}
 Left: A simple tree $\t$ with 5 elementary segments, 2 branching points and 4 leaves. Right: A coding function of $\t$
 with slopes $\pm 1$.}
 \end{center}
 \vspace{-5mm}
 \end{figure}

\subsection{Coding trees with functions}

Let $\sigma\geq 0$ and let $\omega:[0,\sigma]\la \R_+$ be a continuous function such that $\omega(0)=\omega(\sigma)=0$. We define an
equivalence relation $\sim_\omega$ on $[0,\sigma]$ by setting:
$$s\sim_\omega s'\ \hbox{if and only if }\ \omega(s)=\omega(s')= \min_{s\wedge s'\leq r\leq s\vee s'} \omega(r).$$
We let $\t_\omega$ be the quotient space $[0,\sigma]/\!\sim_{\omega}$ and equip $\t_\omega$ with the distance 
induced by
$$d_{(\omega)}(s,s')= \omega(s)+\omega(s')-2 \min_{s\wedge s'\leq r\leq s\vee s'} \omega(r).$$
(notice that $d_{(\omega)}(s,s')=0$ if and only if $s\sim_\omega s'$). By abuse of notation, we keep the same 
notation $d_{(\omega)}$ for the induced distance on $\t_\omega$, and we also let $p_{(\omega)}$ be the
canonical projection from $[0,\sigma]$ onto $\t_\omega$. 

Then it is not hard to verify \cite[Theorem 2.1]{DLG} that 
$(\t_\omega,d_{(\omega)})$ is a tree (called the tree coded by $\omega$) whose root is $p_{(\omega)}(0)=p_{(\omega)}(\sigma)$ by definition. Conversely, any tree can be represented in the form $\t_\omega$. We refer to \cite{Duq} for more information about the coding of trees by functions. 

In what follows we mainly consider $\sigma=1$, 
and we write $\cc$ for the space of all continuous functions $\omega:[0,1]\la \R_+$ such that 
$\omega(0)=\omega(1)=0$. The space $\cc$ is equipped with the distance
$$d_\cc(\omega,\omega')=\|\omega-\omega'\|:= \sup_{t\in[0,1]}|\omega(t)-\omega'(t)|.$$
The mapping $\omega\mapsto \t_\omega$ is continuous, and even Lipschitz, from $\cc$ onto $\T$ (see Lemma 2.3 in \cite{DLG}). 
For any $\omega\in\cc$, we will write $\mu_\omega$ for the probability measure on $\t_\omega$ which is the
pushforward of $\mathbf{1}_{[0,1]}(t)\,\dd t$ under $p_{(\omega)}$. As we will discuss later, 
the mapping $\omega\mapsto (\t_\omega,\mu_\omega)$ is also continuous from $\cc$
into the appropriate space of weighted trees. 

Let $\be=(\be_t)_{0\leq t\leq 1}$ be a normalized Brownian excursion.% and set $\be_t=2\,\be^\circ_t$ for every $t\in[0,1]$. 
We view $\be=(\be_t)_{0\leq t\leq 1}$ as a random variable with values in $\cc$. The random tree 
$(\t_{\be},d_{(\be)})$ is the CRT, see \cite{Al2} or \cite{probasur}  (note that our normalization differs from the one in \cite{Al,Al2}, where the CRT is rather the tree $\t_{2\be}$).
\section{Large deviations in the Gromov-Hausdorff topology}
\label{sec:LDP1}

In this section, we prove Theorem \ref{LDP-GH}. 
Let $H$ denote the subset of $\cc$ defined by saying that $\omega\in H$
if $\omega\in \cc$ and there exists a square integrable function $\dot\omega:[0,1]\la\R$
such that, for every $t\in[0,1]$,
$$\omega(t)=\int_0^t \dot\omega(s)\,\dd s.$$

Recall that $\be$ is a normalized Brownian excursion, which we view as a random variable with values in $\cc$. The following is a variant of the classical Schilder theorem.

\begin{theorem}
\label{Schilder-excu}
The laws 
of $\ve\,\be$, $\ve>0$, satisfy a large deviation principle with speed $\ve^{-2}$ and good rate function
$$I_1(\omega)=\left\{
\begin{array}{ll}
\displaystyle{\frac{1}{2}}\int_0^1 \dot\omega(t)^2\,\dd t\quad&\hbox{if }\omega\in H,\\
\noalign{\smallskip}
\infty&\hbox{otherwise.}
\end{array}
\right.
$$
\end{theorem}

Although Theorem \ref{Schilder-excu} is certainly known to specialists, we were unable to find a precise reference,
and for this reason we sketch a short proof in the appendix below. 

%Let $\T$ be the space of all pointed compact $\R$-trees, equipped with the Gromov-Hausdorff topology.
%With every $\omega\in\cc$ we can associate a (rooted) tree $\t_\omega\in\T$, which is the tree coded by $\omega$.
%Write $\Phi$ for the mapping $\omega\la\t_\omega$. Then $\Phi$ is continuous, and even Lipschitz
%(see Lemma 2.3 in \cite{DLG}.
%
%We also need to introduce the length measure on a tree $\t\in\T$. To this end, let 
%$(x_n)_{n\geq 1}$ be a dense sequence in $\t$, and
%for every $n\geq1$, let
%$$\t_{(n)}=\llbracket \rho_\t,x_1\rrbracket \cup \llbracket \rho_\t,x_2\rrbracket \cup \cdots \cup \llbracket \rho_\t,x_n\rrbracket ,$$
%where $\rho_\t$ is the root of $\t$ and $\llbracket \rho_\t,x\rrbracket$ stands for the 
%geodesic segment from $\rho_\t$ to $x$. Since $\t_{(n)}$ is a finite union of segments, we can define a length measure on $\t_{(n)}$,
%which we denote by $\lambda_{(n)}$. The length measure on $\t$ is
%$$\lambda_\t:=\lim_{n\to\infty} \uparrow \lambda_{(n)}.$$
%The measure $\lambda_\t$ is characterized by the fact that, for any segment $I$ of $\t$,
%$\lambda_\t(I)$
%is the length of $I$ (this shows that $\lambda_\t$ does not depend on the choice of the sequence $(x_n)_{n\geq 1}$.
%By definition, the (finite or infinite) length of $\t$ is $L(\t)=\lambda_\t(\t)$. 

Since the mapping $\omega\mapsto \t_\omega$ is continuous, Theorem \ref{Schilder-excu} and an application of the contraction principle
(Theorem 4.2.1 in \cite{DZ})
immediately show that the laws of $\t_{\ve\be}$ satisfy a large deviation principle with speed $\ve^{-2}$ and good
rate function
$$I(\t)=\inf\{I_1(\omega):\omega \in H,\,\t_\omega=\t\},$$
where $\inf\varnothing =\infty$. So the proof of Theorem \ref{LDP-GH} boils down to verifying that the infimum in the last display is equal to $\frac{1}{2}\,L(\t)^2$.
This is done in the following proposition.

\begin{proposition}
\label{rate-funct}
Let $\t\in\T$. Then the following are equivalent.
\begin{itemize}
\item[\rm(i)] $L(\t)<\infty$;
\item[\rm(ii)] there exists $\omega\in H$ such that $\t_\omega=\t$.
\end{itemize}
If these properties hold, we have
$$\inf\Big\{\int_0^1 \dot\omega(t)^2\,\dd t : \omega\in H,\,\t_\omega=\t\Big\}=4\,L(\t)^2.$$
\end{proposition}

\proof (i)$\Rightarrow$(ii) Consider the trees $\t_{(n)}$ defined by \eqref{simple-approx} from a dense sequence $(x_n)_{n\geq 1}$ in $\t$. For every $n$, we can find a coding function 
$\varphi_n:[0,2L(\t_{(n)})]\la \R_+$ of $\t_{(n)}$, which is piecewise linear with slope $+1$ or $-1$ (see Fig.~\ref{t-coding}). We set $\omega_n(t)=
\varphi_n(2L(\t_{(n)})t)$, for $t\in[0,1]$. Then $\omega_n\in\cc$ and $\omega_n$ is still a coding function of $\t_{(n)}$. Furthermore,
$$|\omega_n(t)-\omega_n(s)|\leq 2L(\t_{(n)})\,|t-s|$$
for every $s,t\in[0,1]$. Since $L(\t_{(n)})\leq L(\t)<\infty$, the sequence $(\omega_n)$ is uniformly equicontinuous, and by taking a 
subsequence we may assume that $\omega_n\la \omega_\infty$ in $\cc$. Since the mapping $\omega\mapsto\t_\omega$ is continuous, we get that $\t_{(n)}=\t_{\omega_n}$
converges to $\t_{\omega_\infty}$ (along the chosen subsequence). But we also know that $\t_{(n)}$ converges to 
$\t$ in the Gromov-Hausdorff sense (in fact for the Hausdorff distance), and thus we get that $\t=\t_{\omega_\infty}$. 
Moreover we have
$$|\omega_\infty(t)-\omega_\infty(s)|\leq 2L(\t)\,|t-s|$$
for every $s,t\in[0,1]$. This implies that $\omega_\infty\in H$, and $|\dot\omega_\infty|\leq 2L(\t)$,
so that
\begin{equation}
\label{tec01}
\int_0^1 \dot\omega_\infty(t)^2\,\dd t \leq 4\,L(\t)^2.
\end{equation}
(ii)$\Rightarrow$(i) Let $\omega\in H$ such that $\t_\omega=\t$. Consider again the sequence $(\t_{(n)})_{n\geq 1}$. For fixed $n$, 
let $\llbracket x,y\rrbracket$ be an elementary segment of $\t_{(n)}$. Without loss of generality, we can assume that $y$ is a 
descendant of $x$ in $\t_{(n)}$, and thus also in $\t$. Set
\begin{align*}
&v=\inf\{t\in[0,1]: p_{(\omega)}(t)=y\}\\
&u=\sup\{t\in[0,v]:p_{(\omega)}(t)=x\}
\end{align*}
Note that $p_{(\omega)}(t)\in \rrbracket x,y\llbracket$ for every $t\in (u,v)$. 
Then,
$$\lambda_{\t_n}(\llbracket x,y\rrbracket)=\lambda_\t(\llbracket x,y\rrbracket)=d_\t(x,y)=d_\t(\rho_\t,y)-d_\t(\rho_\t,x)=\omega(v)-\omega(u),$$
and thus 
\begin{equation}
\label{tec02}
\lambda_{(\t_n)}(\llbracket x,y\rrbracket)=\int_u^v \dot\omega(t)\,\dd t\leq \int_u^v |\dot\omega(t)|\,\dd t.
\end{equation}
Similarly, we can set
\begin{align*}
&v'=\sup\{t\in[0,1]: p_{(\omega)}(t)=y\}\\
&u'=\inf\{t\in[v',1]:p_{(\omega)}(t)=x\}
\end{align*}
and the same argument gives
\begin{equation}
\label{tec03}
\lambda_{(\t_n)}(\llbracket x,y\rrbracket)\leq \int_{v'}^{u'}|\dot\omega(t)|\,\dd t.
\end{equation}
The intervals $(u,v)$ and $(v',u')$ are disjoint. Moreover, when $\llbracket x,y\rrbracket$ varies among
all elementary segments of $\t_{(n)}$, the corresponding intervals $(u,v)$ and $(v',u')$ are also disjoint. By summing 
the bounds \eqref{tec02} and \eqref{tec03} over 
all choices of $\llbracket x,y\rrbracket$, we get 
$$2\lambda_{\t_n}(\t_n)\leq \int_0^1|\dot\omega(t)|\,\dd t,$$
and thus 
$$L(\t_{(n)})\leq \frac{1}{2}\int_0^1|\dot\omega(t)|\,\dd t.$$
Then, by letting $n\to\infty$, we get
$$L(\t)\leq \frac{1}{2}\int_0^1|\dot\omega(t)|\,\dd t<\infty.$$
Moreover, we have
$$\int_0^1 (\dot\omega(t))^2\,\dd t\geq \Big(\int_0^1|\dot\omega(t)|\,\dd t\Big)^2\geq 4L(\t)^2.$$
Together with \eqref{tec01}, this gives the last statement of the proposition. 
This completes the proof \endproof

\rem It is interesting to compare Theorem \ref{LDP-GH} with the known formula for the finite-dimensional
marginals of the tree $\t_{\be}$. Let $n\geq 1$, and consider the subtree $\t_{(n)}$ of the CRT $\t_{\be}$ defined by
formula \eqref{simple-approx} when $x_1,x_2,\ldots,x_n$ are chosen independently according to the mass
measure $\mu_{\be}$. The simple tree $\t_{(n)}$ can be described by a discrete branching structure, which is a binary
tree with $n$ labelled  leaves whose vertices are in one-to-one correspondence with the
elementary segments of $\t_{(n)}$, and a collection of real marks assigned to the vertices of this binary tree
(the mark assigned to a vertex is the length of the corresponding elementary segment). Then,
the distribution 
of $\t_{(n)}$ has density proportional to $L(\t)\,\exp(-2\,L(\t)^2)$ with respect to the uniform measure
on simple trees with binary branching and $n$ leaves (not counting the root as a leaf here): this uniform measure is obtained by summing over
discrete binary structures, and then choosing the marks independently according to Lebesgue measure on $\R_+$ (see the end of Chapter 3 in \cite{Zurich}). 
In a very informal way, one would like to interpret the law of the CRT as having density $L(\t)\,\exp(-2\,L(\t)^2)$ with respect
to a uniform measure on rooted compact $\R$-trees with binary branching, but of course the latter uniform measure does not make sense.

\section{Analytic lemmas}
\label{sec:anal}

In this section, we state and prove two lemmas that will be useful in the proof of Theorem \ref{Schilder-measure}.
The proofs are simple exercises of measure theory, but we provide details for the sake of completeness. 

\begin{lemma}
\label{analytic-lem1} 
Let $a>0$ and let $p:[0,a]\la \R_+$ be a nonnegative Borel function such that $\int_0^a p(t)\,\dd t<\infty$. 
For every $t\in[0,a]$, set
$$f(t)=\int_0^t p(r)\,\dd r.$$
Set $h=f(a)$ and assume that $h>0$.  Let $\mu$ be the pushforward of the measure $\mathbf{1}_{[0,a]}(t)\,\dd t$ under $f$,
and consider the Lebesgue decomposition 
of the measure $\mu$,
$$\mu(\dd s)=\mathbf{1}_{[0,h]}(s)\,q(s)\,\dd s + \theta(\dd s)$$
where $q:[0,h]\la \R_+$ is a nonnegative Borel function, and $\theta$ is singular with respect
to Lebesgue measure on $[0,h]$. Then,
$$\int_0^a p(t)^2\,\dd t=\int_0^h \frac{\dd s}{q(s)}.$$
\end{lemma}

\proof
We 
set
$$g(s)=\inf\{t\in[0,a]:f(t)>s\},$$
for $s\in [0,h)$,
and $g(h)=a$. Then we have $\mu([0,s])=g(s)$, for every $s\in [0,h]$.
We first observe that for every nonnegative Borel function $\varphi:[0,a]\la \R_+$, we have
\begin{equation}
\label{key1}
\int_0^h \varphi\circ g(s)\,\dd s= \int_0^a \varphi(t)\,p(t)\,\dd t.
\end{equation}
This is trivial if $\varphi=1$, and in the case where $\varphi=\mathbf{1}_{[0,r)}$, with $r\in(0,a)$, this is a consequence 
of the fact that, for every $s\in[0,h]$, we have $g(s)<r$ if and only if $f(r)>s$. The general case then follows by 
standard arguments. 

By a standard differentiation theorem (cf.~\cite[Theorem 8.17]{Ru}), we know that $f$ is Lebesgue a.e. differentiable on $[0,a]$, and
$f'(t)=p(t)$, for a.e. $t$. Similarly, we know from \cite[Theorem 8.18]{Ru} that $g$ is also differentiable a.e. on $[0,h]$. We claim that
\begin{equation}
\label{diff-inverse}
g'(s)=\frac{1}{p\circ g(s)},\quad \hbox{a.e.}
\end{equation}
Let us provide a brief justification of \eqref{diff-inverse}. Let $D$ be the (full measure) set of all $t\in[0,a]$ such that
$f'(t)$ exists and is equal to $p(t)$. If $t_0\in D$, we have
$$\lim_{t\to t_0} \frac{f(t)-f(t_0)}{t-t_0} = p(t_0).$$
Note that $g$ is monotone increasing. If $s_0\in[0,h]$ is such that $g(s_0)\in D$, if follows from the previous display that
$$\lim_{s\to s_0}\frac{f(g(s))-f(g(s_0))}{g(s)-g(s_0)}=p(g(s_0)),$$
unless possibly if $g$ is discontinuous at $s_0$, which excludes an at most countable collection of values of $s_0$. 
By definition, $f(g(s))=s$, $f(g(s_0))=s_0$, and, excluding the case where $s_0$ is a jump time of $g$, we get
$$g'(s_0)=\lim_{s\to s_0} \frac{g(s)-g(s_0)}{s-s_0}= \frac{1}{p\circ g(s_0)}.$$
To complete the proof of \eqref{diff-inverse}, we just have to verify that $\{s_0\in [0,h]: g(s_0)\in D\}$ has full Lebesgue 
measure. This follows by taking $\varphi=\mathbf{1}_{D^c}$ in \eqref{key1}.

From the equality $\mu([0,s])=g(s)$, the absolutely continuous part of $\mu$ is $g'(s)\,\dd s$ (cf.~\cite[Theorem 8.18]{Ru}), and thus
$$q(s)=g'(s)=\frac{1}{p\circ g(s)}\,,\quad \hbox{a.e.}$$
Finally, applying \eqref{key1} with $\varphi(t)=p(t)$, we get
$$\int_0^a p(t)^2\,\dd t=\int_0^h p\circ g(s)\,\dd s=\int_0^h \frac{\dd s}{q(s)},$$
which completes the proof \endproof

\begin{lemma}
\label{analytic-lem2}
Let $h>0$ and let $\mu$ be  a finite measure on $[0,h]$, whose Lebesgue decomposition is 
$$\mu(\dd s)=\mathbf{1}_{[0,h]}(s)\,q(s)\,\dd s + \theta(\dd s)$$
where we assume that the (nonnegative Borel) function $q$ is such that $q(s)>0$ a.e. Also set $a=\mu([0,h])>0$. Then, there exists a 
nonnegative Borel function $p:[0,a]\la \R_+$ such that $\int_0^a p(r)\,\dd r=h$ and, for every $s\in[0,h)$,
$$\mu([0,s])=\inf\{t\in[0,a]:\int_0^t p(r)\,\dd r>s\}.$$
\end{lemma}

\proof Consider first the case $\theta=0$, and set for every $s\in[0,h]$,
$$f(s)=\mu([0,s])=\int_0^s q(r)\,\dd r.$$
Also set, for $t\in[0,a]$,
$$g(t)=\inf\{s\in[0,h]:f(s)=t\}.$$
Since  $q(t)>0$ a.e., $f$ is an increasing homeomorphism from $[0,h]$ onto $[0,a]$, and $g$ is the inverse bijection. 
By formula \eqref{key1} (with $p$ replaced by $q$), we have, for every nonnegative Borel function $\varphi:[0,a]\la \R_+$, 
\begin{equation}
\label{key2}
\int_0^a \varphi\circ g(r)\,\dd r= \int_0^h \varphi(r)\,q(r)\,\dd r.
\end{equation}
Set ${\displaystyle p(r)=\frac{1}{q\circ g(r)}}$ for $r\in[0,a]$. Fix $t\in[0,a]$ and apply \eqref{key2} with ${\displaystyle\varphi(r)=\mathbf{1}_{[0,g(t))}(r)\,\frac{1}{q(r)}}$. We get
$$\int_0^t p(r)\,\dd r=\int_0^a \mathbf{1}_{\{g(r)<g(t)\}} \,\frac{1}{q\circ g(r)}\,\dd r=\int_0^h \mathbf{1}_{\{r<g(t)\}} \,\frac{q(r)}{q(r)}\,\dd t=g(t).$$
Finally, since $\mu([0,s])=f(s)=\inf\{t\in[0,a]:g(t)>s\}$, for every $s\in[0,h)$, we get the desired formula. 

Consider then the general case. We keep the same notation $f(s)=\mu([0,s])$ and now set $g(t)=\inf\{s\in[0,h]:f(s)\geq t\}$ for $t\in[0,a]$. Again, 
$g(a)=h$, $g$ is continuous and
nondecreasing (but not a homeomorphism if $\theta$ has atoms), and we have $f(s)=\inf\{t\in[0,a]:g(t)>s\}$ for every $s\in[0,h)$. So the proof will be
complete if we can verify that $g$ is absolutely continuous. 

To this end, set $\mu_0(\dd s)=\mathbf{1}_{[0,h]}(s)\,q(s)\,\dd s$ and $a_0=\mu_0([0,h])$, and, for every $t\in[0,a_0]$,
$$g_0(t)=\inf\{s\in[0,h]:\mu_0([0,s])=t\}.$$
Note that $g_0(a_0)=h=g(a)$.
By the first part of the proof, we know that $g_0$ is absolutely continuous. We can then easily show that the same property holds for $g$. Indeed, let
$x,y\in[0,a]$ with $x<y$. Recalling that $g_0$ is a bijection, we 
can find $\tilde x,\tilde y\in[0,a_0]$ such that $\tilde x\leq \tilde y$ and $g_0(\tilde x)=g(x)$, $g_0(\tilde y)=g(y)$.
It follows that
\begin{align*}
\tilde y-\tilde x&=\mu_0([0,g_0(\tilde y)])-\mu_0([0,g_0(\tilde x)])\\
&=\mu_0([0,g(y)])-\mu_0([0,g(x)])\\
&=\mu_0((g(x),g(y)))\\
&\leq \mu((g(x),g(y)))\\
&=\mu([0,g(y)))-\mu([0,g(x)])\\
&\leq y-x.
\end{align*}
Let $\ve>0$. Since $g_0$ is absolutely continuous, we can find $\delta>0$ such that, whenever $(r_i,s_i)_{i\in I}$
are disjoint subintervals of $[0,h_0]$ with $\sum_{i\in I} (s_i-r_i)<\delta$ we have also $\sum_{i\in I} (g_0(s_i)-g_0(r_i))<\ve$. 
Suppose then that $(x_i,y_i)_{i\in I}$ are disjoint subintervals of $[0,h]$ such that
$\sum_{i\in I} (y_i-x_i)<\ve$. By the preceding considerations, we can find disjoint subintervals $(\tilde x_i,\tilde y_i)_{i\in I}$ of $[0,h_0]$
such that, for every $i\in I$, $g_0(\tilde x_i)=g(x_i)$, $g_0(\tilde y_i)=g(y_i)$, and $\tilde y_i-\tilde x_i\leq y_i-x_i$. In particular,
$\sum_{i\in I} (\tilde y_i-\tilde x_i)\leq \sum_{i\in I} (y_i-x_i)<\delta$ and 
$$\sum_{i\in I} (g(y_i)-g(x_i))=\sum_{i\in I} (g_0(\tilde y_i)-g_0(\tilde  x_i)) <\ve$$
by our choice of $\delta$. This shows that $g$ is absolutely continuous and completes the proof. \endproof

\section{Large deviations in the Gromov-Hausdorff-Prohorov topology}

\label{sec:LDP2}

Recall our notation $\cc$ for the set of all continuous function $\omega:[0,1]\la \R_+$ such that $\omega(0)=\omega(1)=0$. With 
any $\omega\in\cc$, we can associate the tree $\t_\omega$ coded by $\omega$, and the probability measure $\mu_\omega$ on $\t_\omega$
defined as the pushforward of $\mathbf{1}_{[0,1]}(t)\,\dd t$ under the mapping $t\mapsto p_{(\omega)}(t)$. 

We write $\T_{\w}$ for the space of all pairs $(\t,\mu)$ where $\t$ is a tree and $\mu$ is a probability measure on $\t$ (again two pairs $(\t,\mu)$ and $(\t',\mu')$ are identified 
if there exists a root-preserving isometry $F$ from $\t$ onto $\t'$ such that the pushforward of $\mu$ under $F$ is $\mu'$). The space 
$\T_{\w}$ is equipped with the
Gromov-Hausdorff-Prohorov topology (see e.g.~\cite[Section 6]{Mie}).

The following lemma is well known, but we provide a proof for the sake of completeness. 

\begin{lemma}
\label{contiGHP}
The mapping $\omega\mapsto (\t_\omega,\mu_\omega)$ is continuous from $\cc$ into $\T_{\w}$. 
\end{lemma}

\proof Let $\omega,\omega'\in\cc$ and let $\r$ be the correspondence between $\t_\omega$ and $\t_{\omega'}$ defined by
$$\r=\{(p_{(\omega)}(t),p_{(\omega')}(t)):t\in[0,1]\}.$$
Also let $\nu$ be the probability measure on $\t_\omega\times \t_{\omega'}$ defined as the pushforward of $\mathbf{1}_{[0,1]}(t)\,\dd t$ under the mapping $t\mapsto (p_{(\omega)}(t),p_{(\omega')}(t))$.
Trivially, $\nu$ is supported on $\r$, and the pushforward of $\nu$ under the projection $(x,y)\mapsto x$ (resp. the projection $(x,y)\mapsto y$) is $\mu_\omega$ (resp. $\mu_{\omega'}$). Then 
Proposition 6 in \cite{Mie} shows that the Gromov-Hausdorff-Prohorov distance between $(\t_\omega,\mu_\omega)$ and $(\t_{\omega'},\mu_{\omega'})$ is bounded above by $1/2$ times the distortion of $\r$ (\cite{Mie} does not deal with {\it pointed} spaces, but the proof in that case is exactly the same). Since the distortion of $\r$ is bounded above by $4\|\omega-\omega'\|$
(see Lemma 2.3 in \cite{DLG}), the desired result follows.
\endproof

We now turn to the proof of Theorem \ref{Schilder-measure}. To simplify notation, we write $\Theta(\omega)=(\t_\omega,\mu_\omega)$ for $\omega\in\cc$. Since $\Theta$ is continuous (Lemma \ref{contiGHP}), Theorem \ref{Schilder-excu} and an application of the contraction principle show that the
law of $(\t_{\ve\be}, \mu_{\ve \be})=\Theta(\ve\be)$ also satisfies a large deviation principle with good rate function
\begin{equation}
\label{rate-measure}
J(\t,\mu)=\inf\{I_1(\omega):\omega\in\cc,\,\Theta(\omega)=(\t,\mu)\},
\end{equation}
where $\inf\varnothing=\infty$. So the issue is to compute $J(\t,\mu)$. Note that we can replace $\omega\in\cc$ by $\omega\in H$ in the last display, since
$I_1(\omega)=\infty$ if $\omega\notin H$.  In what follows, we systematically exclude the case $\t=\{\rho_\t\}$ where \eqref{rate-measure}
immediately gives $J(\t,\delta_{\rho_\t})=0$.  We then start with a simple case.

\begin{proposition}
\label{rate-simple}
Let $h>0$ and let $\mu$ be a probability measure on $[0,h]$. Suppose that $\t$ is the segment $[0,h]$ rooted at $0$. Then
$$J(\t,\mu)=2\int_0^h \frac{\dd s}{q(s)},$$
where $q(s)\,\dd s$ is the absolutely continuous part in the Lebesgue decomposition of the measure $\mu$
(with respect to Lebesgue measure on $[0,h]$).
\end{proposition}

%\rem By the Cauchy-Schwarz inequality,
%$$h\leq \Big(\int_0^h q(s)\,\dd s\Big)^{12}\Big(\int_0^h \frac{\dd s}{q(s)}\Big)^{1/2} \leq \Big(\int_0^h \frac{\dd s}{q(s)}\Big)^{1/2} ,$$
%so that $J(\t,\mu)\geq 2h^2=2\,L(\t)^2$. 

\proof
Let us first assume that there exists a function $\omega\in H$ such that $\Theta(\omega)=(\t,\mu)$.
The fact that $\omega$ is a coding function of $\t=[0,h]$ implies that there exists $a\in(0,1)$ such that $\omega(a)=h$ and
$\omega$ is both nondecreasing on $[0,a]$ and nonincreasing on $[a,1]$. Since $\omega\in H$, it follows that
we can find two square integrable nonnegative Borel functions $p_1$ and $p_2$ defined respectively on $[0,a]$ and on $[0,1-a]$, such that
\begin{align*}
&\forall t\in[0,a], \ \omega(t)=\int_0^t p_1(r)\,\dd r\\
&\forall t\in[0,1-a], \ \omega(1-t)=\int_0^t p_2(r)\,\dd r.
\end{align*}
Let $\mu_1$ be the pushforward of $\mathbf{1}_{[0,a]}(t)\,\dd t$ under $\omega$. 
%Set
%$$g_1(s)=\inf\{t\in[0,a]:\int_0^t p_1(r)\,\dd r>s\}=\inf\{t\in[0,a]:\omega(t)>s\},$$
%for every $s\in[0,h)$, 
%and $g_1(h)=a$. Then, we have $g_1(s)=\mu_1([0,s])$. 
By Lemma \ref{analytic-lem1}, we have
$$\int_0^a p_1(t)^2\,\dd t=\int_0^h \frac{\dd s}{q_1(s)},$$
where $\mathbf{1}_{[0,h]}(s)\,q_1(s)\dd s$ is the absolutely continuous part of the measure $\mu_1$. 
%Note in particular that $\mu_1\leq \mu_{\omega}=\mu$
%implies $q_1\leq q$ and therefore $\int_0^h \frac{\dd s}{q(s)}\leq \int_0^h \frac{\dd s}{q_1(s)}<\infty$ (by the last display) showing that if $\int_0^h \frac{\dd s}{q(s)}=\infty$,
%there cannot exist a function $\omega\in H$ such that $\Theta(\omega)=(\t,\mu)$). 

Similarly, if $\mu_2$ is the pushforward of $\mathbf{1}_{[a,1]}(t)\,\dd t$ under $\omega$, we have
$$\int_0^{1-a} p_2(t)^2\,\dd t=\int_0^h \frac{\dd s}{q_2(s)},$$
where $\mathbf{1}_{[0,h]}(s)\,q_2(s)\dd s$ is the absolutely continuous part of the measure $\mu_2$. By combining the last two displays, we get
$$\int_0^1 \dot\omega(t)^2\,\dd t=\int_0^a p_1(t)^2\,\dd t + \int_0^{1-a} p_2(t)^2\,\dd t 
=\int_0^h \frac{\dd s}{q_1(s)} + \int_0^h \frac{\dd s}{q_2(s)}\geq 4\int_0^h \frac{\dd t}{q_1(t)+q_2(t)},$$
by a convexity argument. Since $q=q_1+q_2$, we have obtained
$$\frac{1}{2} \int_0^1 \dot\omega(t)^2\,\dd t\geq 2\int_0^h\frac{\dd s}{q(s)}$$
and it follows that
$$J(\t,\mu)\geq 2\int_0^h \frac{\dd s}{q(s)}.$$
Trivially, this inequality remains valid if there exists no function $\omega\in H$ such that $\Theta(\omega)=(\t,\mu)$, since
we have then $J(\t,\mu)=\infty$. 

To get the reverse inequality, we may assume that $\int_0^h \frac{\dd s}{q(s)}<\infty$, which implies that $q(s)>0$ a.e. We then apply Lemma
\ref{analytic-lem2} to the measure $\frac{1}{2}\mu$, whose total mass is $\frac{1}{2}$. It follows that there exists a
nonnegative Borel function $p:[0,\frac{1}{2}]\la \R_+$ such that $\int_0^{1/2} p(r)\,\dd r=h$ and, for every $s\in[0,h)$,
\begin{equation}
\label{reverse-tec}
\frac{1}{2}\mu([0,s])=\inf\{t\in[0,\frac{1}{2}]:\int_0^t p(r)\,\dd r>s\}.
\end{equation}
We then define $\omega=(\omega(t))_{t\in [0,1]}$ by setting for every $t\in[0,\frac{1}{2}]$,
$$\omega(t)=\omega(1-t)=\int_0^t p(r)\,\dd r.$$
From \eqref{reverse-tec}, we see that $\frac{1}{2}\mu$ is the pushforward of $\mathbf{1}_{[0,1/2]}(s)\dd s$ under $\omega$.

Since $p\geq 0$ and $\omega(\frac{1}{2})=h$, $\omega$ is a coding function of $\t$, and moreover $\omega\in H$
since
$$\int_0^1 \dot \omega(t)^2\,\dd t= 2\int_0^{1/2} p(t)^2\,\dd t= 4 \int_0^h\frac{\dd s}{q(s)}<\infty$$
using Lemma \ref{analytic-lem1} (with $a=\frac{1}{2}$, $f=\omega$, and $\mu$ replaced by $\frac{1}{2}\mu$ ) in the last equality. Furthermore,we have
$\mu_\omega=\mu$, and thus $\Theta(\omega)=(\t,\mu)$. 
We then conclude from
\eqref{rate-measure} and the last display that
$$J(\t,\mu)\leq \frac{1}{2}\int_0^1 \dot \omega(t)^2\,\dd t = 2\int_0^h\frac{\dd s}{q(s)}.$$
This completes the proof. \endproof

We now complete the proof of Theorem \ref{Schilder-measure}.

\smallskip
\noindent{\it First step.} In this first step, we prove formula \eqref{key-formu} when $\t$
is a simple tree and $\mu(\b_\t\cup\ll_\t)=0$.
Let $S$ be an elementary segment of $\t$. Then, we have $S=\llbracket x,y\rrbracket$, where $y$ is a descendant of $x$. Moreover,
both $x$ and $y$ must belong to $\b_\t\cup\ll_\t$. We write $S^\circ=\rrbracket x,y\llbracket$ for the interior of $S$.
Let $\omega$ be a coding function of $\t$. Then, it is not hard to see that there exist two disjoint open subintervals $(t_S,u_S)$ and $(v_S,w_S)$ of $[0,1]$ 
such that the following holds:
\begin{itemize}
\item[$\bullet$] $p_{(\omega)}(t_S)=p_{(\omega)}(w_S)=x$ and $p_{(\omega)}(u_S)=p_{(\omega)}(v_S)=y$;
\item[$\bullet$] the property $p_{(\omega)}(t)\in S^\circ$ holds if and only if
$t\in(t_S,u_S)\cup(v_S,w_S)$;
\item[$\bullet$] the function $\omega$
is monotone nondecreasing on $(t_S,u_S)$ and monotone nonincreasing on $(v_S,w_S)$. 
\end{itemize}

%Let $\mu$ be a probability measure on the simple tree $\t$. Our goal is to prove formula \eqref{key-formu}
%for $J(\t,\mu)$. For simplicity, we consider first the case where $\mu(\b_\t\cup\ll_\t)=0$.
Suppose that $\omega\in H$ is such that $\Theta(\omega)=(\t,\mu)$, and in particular $\t_\omega=\t$. Fix an elementary segment $S=\llbracket x,y\rrbracket$ of  $\t$, then the preceding observations show that the pushforward of Lebesgue measure on $[t_S, u_S]\cup [v_S,w_S]$
under $p_{(\omega)}$
must be the restriction of $\mu_\omega=\mu$ to $\llbracket x,y\rrbracket$. A direct adaptation of the proof of Proposition \ref{rate-simple} 
(we omit the details) implies that
\begin{equation}
\label{LB-simple}
\frac{1}{2}\int_{(t_S,u_S)\cup(v_S,w_S)} \dot\omega(t)^2\,\dd t\geq 2\int_{\llbracket x,y\rrbracket}\frac{\lambda_\t(\dd z)}{q_\t(z)},
\end{equation}
where $q_\t(z)\lambda_\t(\dd z)$ is the absolutely continuous part in the Lebesgue decomposition 
of $\mu$ with respect to the length measure $\lambda_\t$. Let $\S_\t$ be the collection of all elementary segments of $\t$ Observe that
$\t\backslash (\b_\t\cup \ll_\t)$ is the disjoint union of the interiors $S^\circ$ for all $S\in\S_\t$. By summing the preceding lower bound over all choices
of $S\in\S_\t$, we get
\begin{equation}
\label{LB-simple2}
I_1(\omega)=\frac{1}{2}\int_0^1\dot\omega(t)^2\,\dd t\geq 2\int_{\t}\frac{\lambda_\t(\dd z)}{q_\t(z)}.
\end{equation}
The lower bound for $J(\t,\mu)$ in formula \eqref{key-formu} then follows from \eqref{rate-measure}.

To get the corresponding upper bound, we may assume that $\int_0^1 q_\t(z)^{-1}\lambda_\t(\dd z)<\infty$. Then, given an elementary segment $S=\llbracket x,y\rrbracket$ of $\t$
(where as previously $y$ is a descendant of $x$),  we can adapt the end of the proof of Proposition \ref{rate-simple} to
see that  equality holds in \eqref{LB-simple} provided that the following properties are satisfied:
\begin{itemize}
\item $u_S=t_S+\frac{\mu(S)}{2}$, $w_S=v_S+\frac{\mu(S)}{2}$;
\item for every 
$r\in[0,\frac{\mu(S)}{2}]$,
$$\omega(u_S+r)-\omega(u_S)=\omega(w_S-r)-\omega(w_S)=\int_0^r p(t)\,\dd t$$
where the (nonnegative) function $(p_S(t))_{0\leq t\leq \mu(S)/2}$ satisfies $\int_0^{\mu(S)/2} p_S(t)\,\dd t=d_\t(x,y)$, and, for every $r\in [0,d_\t(x,y))$,
$$\frac{1}{2}\mu(\llbracket x,x_r\rrbracket)=\inf\{t\in [0,\frac{\mu(S)}{2}]: \int_0^t p_S(u)\,\dd u>r\},$$
where $x_r$ denotes the unique point of $\llbracket x,y\rrbracket$ at distance $r$ from $x$. 
\end{itemize}

We leave it to the reader to verify that we can construct the coding function $\omega\in H$ so that the preceding
properties hold for every elementary segment $S$, and it follows that equality also holds in \eqref{LB-simple2} in that case. Using \eqref{rate-measure}, we get that 
formula \eqref{key-formu} is verified 
when $\t$ is a simple tree and the measure $\mu$ gives no mass to $\b_\t\cup\ll_\t$. 

\medskip
\noindent{\it Second step}. Still assuming that $\t$ is a simple tree, consider now the general case where $\mu(\b_\t\cup\ll_\t)$ may be nonzero. 
We let $\mu_0$ be the restriction of $\mu$ to $\t\backslash(\b_\t\cup\ll_\t)$ and set
$$\alpha=\mu_0(\t)= 1-\mu(\b_\t\cup\ll_\t).$$
We may exclude the case where $\alpha=0$, because, for any coding function $\omega$ of $\t$, $\mu_\omega$
must give positive mass to $\t\backslash(\b_\t\cup\ll_\t)$ and thus the condition $\mu_\omega(\b_\t\cup\ll_\t)=1$ can never be
satisfied (the formula for $J(\t,\mu)$ is trivial in that case). 

Let $\omega\in H$ be such that $(\t_\omega,\mu_\omega)=(\t,\mu)$. We define $\omega^*_0=(\omega^*_0(t))_{0\leq t\leq \alpha}$ by ``removing'' the flat pieces 
of the graph  of $\omega$ corresponding to points of $\b_\t\cup\ll_\t$. More precisely, for every $s\in[0,\alpha]$,
$$\omega^*_0(s)=\omega\Big(\inf\Big\{ r\in[0,1]:\int_0^r \mathbf{1}_{\t\backslash(\b_\t\cup\ll_\t)}(p_{(\omega)}(t))\,\dd t\geq s\Big\}\Big).$$
We then set $\omega_0(t)=\omega^*_0(\alpha t)$ for every $t\in[0,1]$. It is easy to verify that 
$\omega_0$ is still a coding function of $\t$, and moreover 
$$\mu_{\omega_0}=\frac{\mu_0}{\alpha}.$$
We have $\dot\omega_0(t)=\alpha\,\dot\omega^*_0(\alpha t)$ for $t\in[0,1]$, and therefore
$$\int_0^1 \dot\omega_0(t)^2\,\dd t= \alpha^2\int_0^1 \dot\omega^*_0(\alpha t)^2\,\dd t=\alpha \int_0^\alpha  \dot\omega^*_0(u)^2\,\dd u=\alpha\int_0^1 \dot\omega(u)^2\,\dd u,$$
because obviously $\dot\omega(u)=0$ on the flat pieces of the graph of $\omega$. 

Write $q_0(x)\lambda_\t(\dd x)$ for the absolutely continuous part of $\frac{1}{\alpha}\mu_0(\dd x)$. Then, $q_0=\frac{1}{\alpha}q_\t$, and since
$\mu_0$ puts no mass on $\b_\t\cup\ll_\t$, 
the first step of the proof gives
$$\frac{1}{2} \int_0^1 \dot\omega_0(t)^2\dd t\geq 2\int_\t \frac{\lambda_\t(\dd z)}{q_0(z)}= 2\alpha \int_\t \frac{\lambda_\t(\dd z)}{q_\t(z)}.$$
It follows that we have again
$$\frac{1}{2}\int_0^1\dot\omega(t)^2\,\dd t\geq 2 \int_{\t}\frac{\lambda_\t(\dd z)}{q_\t(z)}.$$
and so we have proved that
$$J(\t,\mu)\geq 2\int_\t \frac{\lambda_\t(\dd z)}{q_\t(z)}.$$

The proof of the upper bound is similar. Thanks to the first step of the proof, we can find $\omega_0\in H$
such that $\Theta(\omega_0)=(\t,\alpha^{-1}\mu_0)$ and $\frac{1}{2}\int_0^1 \dot\omega_0(t)^2\dd t=2\int q_0(z)^{-1} \lambda_\t(\dd z)$. We then define $\omega_0^*(t)=\omega_0(t/\alpha)$
for $0\leq t\leq \alpha$. By inserting ``flat pieces'' in the graph of $\omega^*_0$, we can construct $\omega\in H$ such that $\t_\omega=\t$, $\mu_\omega=\mu$ and $\omega^*_0$ is derived from $\omega$ by
the procedure explained above. Then we have
$$\frac{1}{2}\int_0^1\dot\omega(t)^2\,\dd t=\frac{1}{2\alpha} \int_0^1 \dot\omega_0(t)^2\,\dd t=\frac{2}{\alpha}\int \frac{\lambda_\t(\dd z)}{q_0(z)}= 2\int_\t \frac{\lambda_\t(\dd z)}{q_\t(z)}.$$
This completes the proof of formula \eqref{key-formu} in the case when $\t$ is a simple tree. 

\medskip
\noindent{\it Third step.} Consider now a general pair $(\t,\mu)\in\T_{\w}$ such that $\t$ has finite  length ($\lambda_\t(\t)<\infty$).
With a dense sequence $(x_n)_{n\geq 1}$ in $\t$, 
we can associate the simple trees $\t_{(n)}$ defined by formula \eqref{simple-approx}.
As we already noticed, $\t_{(n)}$ converges to $\t$ as $n\to\infty$
in the Hausdorff sense, and $\lambda_\t(\t\backslash \t_{(n)})$ converges to $0$ as $n\to\infty$.

Consider a connected component $C$ of the open set $\t\backslash \t_{(n)}$. Then, the closure of $C$ is a compact $\R$-tree, which contains a 
unique point $\rho_C$ at minimal distance from $\rho_\t$ (in such a way that $C$ is composed of descendants 
of $\rho_C$). We write $\Delta_n$ for the set of all connected components of $\t\backslash \t_{(n)}$ and we define a 
probability measure $\mu_{(n)}$ on $\t_{(n)}$ by setting
$$\mu_{(n)}=\mu_{|\t_{(n)}} + \sum_{C\in \Delta_n} \mu(C)\,\delta_{\rho_C},$$
where $\mu_{|\t_{(n)}}$ denotes the restriction of $\mu$ to $\t_{(n)}$. Again, it is easy to verify that $\mu_{(n)}$
converges weakly to $\mu$ as $n\to\infty$. Consequently, the pair $(\t_{(n)},\mu_{(n)})$ converges to $(\t,\mu)$
in the Gromov-Hausdorff-Prohorov sense, and since we know that the rate function $J$ is lower-semicontinuous,
we have
\begin{equation}
\label{lower-SC}
J(\t,\mu)\leq \liminf_{n\to\infty} J(\t_{(n)},\mu_{(n)}).
\end{equation}
The case of simple trees gives
$$J(\t_{(n)},\mu_{(n)})=2\int_\t \frac{1}{q_n(x)}\,\lambda_{\t_{(n)}}(\dd x),$$
where $q_n(x)\lambda_n(\dd x)$ is the absolutely continuous part of $\mu_{(n)}(x)$. Note that $\lambda_{\t_{(n)}}$
is the restriction of $\lambda_\t$ to $\t_{(n)}$. From the definition of
$\mu_{(n)}$ it is immediate that $q_n$ is just the restriction of $q_\t$ to $\t_{(n)}$. Hence
$$J(\t_{(n)},\mu_{(n)})=2\int_{\t_{(n)}}\frac{1}{q_\t(x)}\,\lambda_\t(\dd x) \build{\la}_{n\to\infty}^{} 2\int_\t \frac{1}{q_\t(x)}\,\lambda_\t(\dd x),$$
using the fact that $\mathbf{1}_{\t_{(n)}}(x) \uparrow 1$, $\lambda_\t(\dd x)$ a.e., since we know that
$\lambda_\t(\t\backslash\t_{(n)})\downarrow 0$ as $n\to\infty$. From the last display and \eqref{lower-SC}, we conclude that
$$J(\t,\mu)\leq 2\int_\t \frac{1}{q_\t(x)}\,\lambda_\t(\dd x).$$

In order to prove the corresponding lower bound for $J(\t,\mu)$, let $\omega\in H$
such that $\Theta(\omega)=(\t,\mu)$. For every $n\geq 1$, we observe that, for every $C\in \Delta_n$, the set 
$\{t\in[0,1]:p_{(\omega)}(t)\in C\}$ must be an open interval $(u_C,v_C)$ such that $\omega(u_C)=\omega(v_C)$ and 
$\omega(t)>\omega(u_C)$ for $t\in(u_C,v_C)$. Write $\jj_n$
for the union of all intervals $(u_C,v_C)$ for $C\in \Delta_n$, and define, for $t\in[0,1]$,
$$\omega_n(t)=\left\{
\begin{array}{ll}
\omega(t)&\hbox{if }t\in[0,1]\backslash \jj_n\\
\omega(u_C)\quad&\hbox{if }t\in (u_C,v_C),\; \hbox{with } C\in \Delta_n.
\end{array}
\right.
$$
Then $\omega_n\in\cc$. Moreover, 
we claim that $(\t_{\omega_n},\mu_{\omega_n})=(\t_{(n)},\mu_{(n)})$ (in the sense that $(\t_{\omega_n},\mu_{\omega_n})$ and $(\t_{(n)},\mu_{(n)})$
correspond to the same element of $\T_\w$). To justify the equality $\t_{\omega_n}=\t_{(n)}$, note that,
for every $r,s\in[0,1]$, the property $p_{(\omega_n)}(r)=p_{(\omega_n)}(s)$ holds if either $r$ and $s$ belong to the same interval $(u_C,v_C)$, 
for some $C\in 
\Delta_n$, or both $r$ and $s$ belong to $[0,1]\backslash \jj_n$ and $p_{(\omega)}(r)=p_{(\omega)}(s)$. Thanks to this
observation, we can construct a mapping $\psi_n$ from $\t_{\omega_n}=[0,1]/\!\sim_{\omega_n}$ onto $\t_{(n)}$ that maps the equivalence class
of an interval $(u_C,v_C)$ to $p_{(\omega)}(u_C)=\rho_C$, and every point $p_{(\omega_n)}(r)$ with $r\in [0,1]\backslash \jj_n$ to $p_{(\omega)}(r)$. It is easy to verify
that $\psi_n$ is an isometry, which gives the equality $\t_{\omega_n}=\t_{(n)}$. The identification of $\mu_{\omega_n}$ with $\mu_{(n)}$ follows by
similar arguments.

Then, it is straightforward to verify that,
for every $t\in[0,1]$,
$$\omega_n(t)=\int_0^t \mathbf{1}_{[0,1]\backslash \jj_n}(s)\,\dot\omega(s)\,\dd s$$
and thus $\omega_n\in H$ and 
$$\dot\omega_n(t)=\mathbf{1}_{[0,1]\backslash \jj_n}(t)\,\dot\omega(t).$$
In particular,
$$\frac{1}{2} \int_0^1 \dot\omega(t)^2\,\dd t\geq \frac{1}{2} \int_0^1 \dot\omega_n(t)^2\,\dd t\geq J(\t_{(n)},\mu_{(n)})
\build{\la}_{n\to\infty}^{} 2\int_\t \frac{1}{q_\t(x)}\,\lambda_{\t}(\dd x),$$
and, from \eqref{rate-measure}, we get
$$J(\t,\mu)\geq 2\int_\t \frac{1}{q_\t(x)}\,\lambda_\t(\dd x).$$
This completes the proof of Theorem \ref{Schilder-measure}. \endproof

\section*{Appendix: Proof of Theorem \ref{Schilder-excu}}

Let $\ddd$ denote the space of all continuous functions $\omega:[0,1]\la\R$ such that $\omega(0)=\omega(1)=0$, which is equipped with the supremum norm.
Also let $H'$ denote the subset of $\ddd$ consisting of all functions
$\omega\in \ddd$ such that there exists a square integrable function $\dot\omega:[0,1]\la\R$
such that $\omega(t)=\int_0^t \dot\omega(s)\,\dd s$, for every $t\in[0,1]$.

Write $\bb=(\bb_t)_{0\leq t\leq 1}$ for a standard Brownian bridge, which may be viewed as a random variable with values in $\ddd$. 
Then the laws 
of $\ve\,\bb$, $\ve>0$, in $\ddd$ satisfy a LDP (large deviation principle) 
 with speed $\ve^{-2}$ and good rate function
$$I_0(\omega)=\left\{
\begin{array}{ll}
\displaystyle{\frac{1}{2}}\int_0^1 \dot\omega(t)^2\,\dd t\quad&\hbox{if }\omega\in H',\\
\noalign{\smallskip}
\infty&\hbox{otherwise.}
\end{array}
\right.
$$
This is an immediate consequence of the classical Schilder theorem, noting that we may construct $\bb$ as $\bb_t=B_t-t\,B_1$, where $B$
is a standard linear Brownian motion, and the mapping $\omega\mapsto (\omega(t)-t\omega(1))_{0\leq t\leq 1}$ is continuous for the supremum norm. 

To relate the Brownian excursion to the Brownian bridge, we use the Vervaat transformation \cite{Ver}.
We define a mapping $\Gamma$ from $\ddd$ into $\cc$ as follows. If $\omega\in \ddd$, we
define $\Gamma(\omega)\in\cc$ by  $\Gamma(\omega)(t)=\omega(\{t_{min}+t\})-\omega(t_{min})$, where $t_{min}:=\min\{t\in[0,1]:\omega(t)=\min_{s\in[0,1]}\omega(s)\}$
and $\{x\}$ denotes the fractional part of $x$. Then, the Brownian excursion $\be$ may be constructed from the bridge $\bb$ by
$\be=\Gamma(\bb)$. However, we cannot directly apply the contraction principle because $\Gamma$ is not continuous on $\ddd$. Still we can proceed as
follows.

Let $K$ be a compact subset of $\cc$. Then, 
$$\P(\ve\be\in K)=\P(\Gamma(\ve\bb)\in K)=\P(\ve \bb\in \Gamma^{-1}(K))\leq \P(\ve \bb\in \ov{\Gamma^{-1}(K)}),$$
where we use the notation $\ov{A}$ for the closure of a set $A\subset\ddd$. By the LDP for the Brownian bridge, we have thus
\begin{equation}
\label{Schi3}\limsup_{\ve\to 0} \ve^2\log\P(\ve\be\in K)\leq -\inf_{\omega\in \ov{\Gamma^{-1}(K)}}I_0(\omega). 
\end{equation}
We now claim that
\begin{equation}
\label{Schi4}
\inf_{\omega\in \ov{\Gamma^{-1}(K)}}I_0(\omega)=\inf_{\omega\in \Gamma^{-1}(K)}I_0(\omega).
\end{equation}
To see this, let $(\omega_n)_{n\geq 1}$ be a sequence in $\Gamma^{-1}(K)$ that converges to $\omega_\infty$ in $\ddd$. We need to verify that
\begin{equation}
\label{Schi5}
I_0(\omega_\infty)\geq \inf_{\omega\in \Gamma^{-1}(K)}I_0(\omega).
\end{equation}
Write $\w_n=\Gamma(\omega_n)\in K$. Up to extracting a subsequence, we can assume that $\w_n$ converges to $\w_\infty$ in $\cc$ and $\w_\infty\in K$. We also know
that 
$\w_n(t)=\omega_n(\{t^n_{min}+t\})-\omega_n(t^n_{min})$, where $t^n_{min}\in[0,1]$ achieves the minimum of $\omega_n$. By extracting another subsequence if necessary, we can assume that
$t^n_{min}\la t_\infty$ as $n\to\infty$, and we get that the limit $\w_\infty$ of the sequence $(\w_n)$ is
$$\w_\infty(t)=\omega_\infty(\{t_{\infty}+t\})-\omega(t_{\infty})$$
and $\omega_\infty(t_\infty)=\min_{[0,1]}\omega_\infty$. Finally, one immediately verifies that $I_0(\w_\infty) = I_0(\omega_\infty)$, and moreover the (trivial) equality 
$\Gamma(\w_\infty)=\w_\infty$ gives $\w_\infty\in \Gamma^{-1}(K)$. The bound \eqref{Schi5} follows, which also gives \eqref{Schi4}. Finally, the LDP upper bound for compact sets follows from
\eqref{Schi3} and \eqref{Schi4}, noting that
$$\inf_{\omega\in \Gamma^{-1}(K)}I_0(\omega)=\inf_{\tilde\omega\in K}\Big(\inf\{I_0(\omega):\omega\in\ddd, \Gamma(\omega)=\tilde\omega\}\Big)=\inf_{\tilde\omega\in K} I_1(\tilde\omega).$$

Let us turn to the LDP lower bound for open sets. We write $\cc^*$ for the set of all $\omega\in\cc$ such that $\omega(t)>0$ for every $t\in(0,1)$, and $\ddd^*$ for the union of $\cc^*$
and of the set of all $\omega\in\ddd$ whose minimum is negative and is achieved at a unique point (necessarily in $(0,1)$). Note that $\Gamma$ maps $\ddd^*$ onto $\cc^*$, and $\ddd\backslash\ddd^*$ onto $\cc\backslash \cc^*$. 
Furthermore, $\Gamma$ is continuous at every point of $\ddd^*$. 

Let $G$ be an open subset of $\cc$. Then,
$$\P(\ve\be\in G)=\P(\ve\be\in G\cap \cc^*)=\P(\ve\bb\in \Gamma^{-1}(G\cap \cc^*)).$$
Since $\Gamma$ is continuous on $\ddd^*$, we can find an open subset $O$ of $\ddd$
such that $\Gamma^{-1}(G\cap \cc^*)=O\cap \ddd^*$. It follows that
$$\liminf_{\ve\to 0} \ve^2\log\P(\ve\be\in G)=\liminf_{\ve\to 0} \ve^2\log\P(\ve\bb\in O\cap \ddd^*)=\liminf_{\ve\to 0} \ve^2\P(\ve\bb\in O)\geq -\inf_{\omega\in O} I_0(\omega)
$$
and
$$\inf_{\omega\in O} I_0(\omega)\leq \inf_{\omega\in O\cap\ddd^*} I_0(\omega)=\inf_{\tilde \omega\in G\cap\cc^*} I_1(\tilde\omega)=\inf_{\tilde \omega\in G} I_1(\tilde\omega),$$
where the last equality holds because any $\tilde \omega\in G$ can be written as the limit of a sequence $(\tilde\omega_n)$ in $G\cap\cc^*$ such that $I_1(\tilde \omega_n)$
converges to $I_1(\tilde\omega)$. This completes the proof of the LDP upper bound for open sets. 

To complete the proof of Theorem \ref{Schilder-excu}, it only remains to verify the exponential tightness property, meaning that, for every $\alpha>0$, we can find a 
compact subset $K_\alpha$ of $\cc$ such that 
$$\limsup_{\ve \to 0} \ve^2\P(\ve\be\notin K_\alpha)\leq -\alpha.$$
Via the Vervaat transformation, this follows from the analogous result for the Brownian bridge, and we omit the details. \endproof

\medskip
\noindent{\bf Acknowledgments.} I thank Nicolas Curien and Rapha\"el Cerf for useful conversations.


\begin{thebibliography}{99}
\bibitem{Al}
{D. Aldous}, The continuum random tree I. {\it Ann. Probab.} 19, 1--28 (1991)

\bibitem{Al2}
{D. Aldous}, The continuum random tree III. {\it Ann. Probab.} 21, 248--289 (1993)

\bibitem{DZ}
{A. Dembo, O. Zeitouni}, Large Deviations Techniques and Applications. Springer, 2008

\bibitem{Duq}
{T. Duquesne},
The coding of compact real trees by real valued functions. arXiv:math/0604106 

\bibitem{DLG}
{T. Duquesne, J.-F. Le Gall},
Probabilistic and fractal aspects of L\'evy trees. {\it Probab. Theory Related Fields}  131, 553--603 (2005)

\bibitem{Evans}
{S.N. Evans},
Probability and Real Trees. Lecture Notes Math. 1920. Springer, 2007

\bibitem{Zurich}
 {J.-F. Le Gall}, Spatial Branching Processes, Random Snakes and Partial Differential Equations. 
 Birkha\"user, 1999

\bibitem{probasur} 
 {J.-F. Le Gall}, {Random trees and applications}.
{\it Probab. Surveys} {2}, 245--311 (2005)

\bibitem{Mie}
{G. Miermont},
Tessellations of random maps of arbitrary genus.
{\it Ann. Sci. Ec. Norm. Sup.} 42, 725--781 (2009)

\bibitem{Ru}
{W. Rudin},
Real and Complex Analysis. 
McGraw-Hill, 1974

\bibitem{Ver}
{W. Vervaat}, 
A relation between Brownian bridge and Brownian excursion. {\it Ann. Probab.} 10, 234--239 (1982)

\end{thebibliography}
\end{document}